\documentclass[12pt]{amsart}
\usepackage{fourier}
\usepackage[author-year]{amsrefs}
\usepackage{graphics}
\usepackage{hyperref}
\usepackage{qtree}
\usepackage[all]{xy}
\usepackage{enumerate}
\usepackage[mathscr]{eucal}
\usepackage{bbm}
\usepackage{tikz}
\usepackage{parskip}
\providecommand{\U}[1]{\protect\rule{.1in}{.1in}}
\renewcommand{\P}{{\mathbb P}}

\def\theenumi{\arabic{enumi}}

\def\theenumii{\alph{enumii}}
\def\p@enumii{\theenumi.}

\def\theenumiii{\arabic{enumiii}}
\def\p@enumiii{(\theenumi)(\theenumii)}

\def\p@enumiv{\p@enumiii.\theenumiii}

\theoremstyle{plain}
\newtheorem{theorem}{Theorem}[section]

\newtheorem{lemma}[theorem]{Lemma}

\newtheorem{proposition}[theorem]{Proposition}

\newtheorem{corollary}[theorem]{Corollary}

\newtheorem{conjecture}[theorem]{Conjecture}

\numberwithin{equation}{section}

\theoremstyle{definition}

\newtheorem{definition}[theorem]{Definition}
\newtheorem{definitions}[theorem]{Definitions}

\newtheorem{remark}[theorem]{Remark}

\newtheorem{thmab}{Theorem}

\newtheorem{conjab}[thmab]{Conjecture}

\DeclareMathOperator{\FI}{FI}

\DeclareMathOperator{\Hom}{Hom}

\DeclareMathOperator{\Norm}{Norm}
\DeclareMathOperator{\Var}{Var}
\DeclareMathOperator{\Cov}{Cov}

\newcommand{\bvec}[1]{{\boldsymbol #1}}
\newcommand{\G}{{\mathcal G}}
\newcommand{\Sn}{\mathfrak{S}}

\newcommand{\Z}{{\mathbb{Z}}}

\newcommand{\R}{\mathbb{R}}

\newcommand{\E}{\mathcal{E}}

\newcommand{\F}{\mathcal{F}}

\newcommand{\X}{\mathcal{X}}

\newcommand{\arXiv}[1]{\href{http://arxiv.org/abs/#1}{\nolinkurl{arXiv:#1}}}
\newcommand{\arXivV}[2]{\href{http://arxiv.org/abs/#1}{\nolinkurl{arXiv:#1v#2}}}

\title[Random braiding]{A model for random braiding in graph configuration spaces}

\author[D.A.~Levin]{David A.\ Levin}
\address{University of Oregon Department of Mathematics, Fenton Hall, Eugene, OR 97405}
\email{dlevin@uoregon.edu}
\author[E.~Ramos]{Eric Ramos}
\address{University of Oregon Department of Mathematics, Fenton Hall, Eugene, OR 97405}
\email{eramos@uoregon.edu}
\author[B.~Young]{Benjamin Young}
\address{University of Oregon Department of Mathematics, Fenton Hall, Eugene, OR 97405}
\email{bjy@uoregon.edu}

\thanks{The second author was supported by NSF grant DMS-1704811.}

\begin{document}

\begin{abstract}
We define and study a model of winding for non-colliding particles in
finite trees.  We prove that the asymptotic behavior of this statistic
satisfies a central limiting theorem, analogous to similar results on
winding of bounded particles in the plane \cite{WT}.   We also
propose certain natural open questions and conjectures, whose
confirmation would provide new insights on configuration spaces of
trees. 
\end{abstract}


\maketitle

\section{Introduction}

The winding of Brownian particles in the plane is well-studied; 
See \ocites{Sp,Dur,PY,LM,MS} for a few early references.
Two independent Brownian motions in $\R^2$ will not collide almost
surely, so the winding of the particles around each other is a
well-defined process.  In particular, the \emph{difference} between
two processes is a single Brownian particle in the plane, which almost
surely does not hit points.   Let $W(t)$ be the total expended angle traversed around the origin
until time $t$. 
\ocite{Sp} proved that  $W(t)/\log(t)$ converges
in distribution, as $t \to \infty$,
to a Cauchy distribution. Variations of this winding statistic have
been studied:  the discrete version on the lattice \cite{Bel}, allowing 
a repulsive force to the origin, or even changing the plane to some
other surface \cites{LM,MS,Wata}.   In all cases,
one can define a similar winding statistic and investigate the large $t$
scaling limit.  These studies motivated us to investigate the winding
of two diffusive particles in a one-dimensional topological space.

As noted, a pair of planar Brownian particles will collide with probability zero. In particular, running two
independent Brownian motions up to some fixed time $t$ produces a
trajectory in the \emph{configuration space} of the plane.
Given a topological space $X$, the $k$-stranded configuration space is
the topological space
$\F_k(X) := \{(x_1,\ldots,x_n) \in X^n \mid x_i \neq x_j\}/\Sn_n$,
where the action of symmetric group $\Sn_n$ is given by permuting
coordinates. These spaces have a rich history in topology,
mathematical physics, and many other fields. (See \ocite{knudson} for a survey of such results.)
Thus the winding $W(t)$ may be interpreted as a measuring the
quantity of  homology accumulated by the two particle path in
configuration space up to time $t$.
We adopt this interpretation of winding here.

The $k$-particle exclusion process on a graph $G$ is the Markov chain
with state space $\Omega = \{0,1\}^{V(G)}$; a configuration $\omega$ with
$\omega(v)$ = 1 for exactly $v \in S$ can be visualized by placing
a particle on each $v \in S$ (making such vertices \emph{occupied})
and leaving each $v \not\in S$ \emph{unoccupied}.   The chain moves by
selecting an edge at random, and swapping the status
(occupied/unoccupied) of the two endpoints.     This process is
well-studied from multiple points-of-views, e.g.\ 
hydrodynamic limits (see, for a sampling of early work,
\ocites{GPV,KOV,Rez}, and also the book \ocite{KL}),
and mixing time (\ocites{Wil, Morris, Lacoin}).  See also the book
\ocite{Lig1} for
references and other aspects.    The $k$-stranded configuration space of the graph is the identical space;
thus it is natural to study the exclusion process from the topological
point of view, which has, to our knowledge, not been previously
explored.  The purpose of this paper is to describe one aspect of
the topological evolution of the $2$-particle exclusion process,
namely the \emph{accumulative homology}, derive some basic properties
and work out a few specific examples, and to suggest questions for
future study.

\subsection{Statement of main results and conjectures}

Let $G$ be a finite tree, and write $\rho(t)$ for the random path in
$\F_2(G)$ given as the realization of the exclusion process for $t$
steps. In what follows, we also fix a \emph{planar leaf-rooted
  structure} on $G$. That is, a choice of embedding of $G$ into the
plane, as well as a choice of leaf, henceforth called the \emph{root}
of $G$. Importantly, these choices induce a well-ordering on the vertices
of $G$ via a depth-first process originating from the root. We will
see in Section \ref{confBack} that this data further induces the
following: 

\begin{itemize}
  \item A basis for $H_1(\F_2(G)) \cong \Z^g$;
  \item a loop $\widehat{\rho}(t):S^1 \rightarrow \F_2(G)$ naturally
    associated to $\rho(t)$, called the \emph{closure of $\rho(t)$}. 
\end{itemize}

The \emph{winding} of the path $\rho(t)$ is the multivariate
statistic
$W(t) := \widehat{\rho}(t) \in H_1(\F_2(G)) \cong \Z^g$. Just as was
the case in the aforementioned classical settings, our main result
describes the asymptotic behavior of this statistic.  

\begin{thmab}
  Let $G$ be a planar leaf-rooted tree. There exists a $g \times g$
  matrix $\mathbf{\Sigma}$ such that 
  \[
    W(t)/\sqrt{t} \stackrel{\mathcal{D}}
      \rightarrow \Norm(\mathbf{\Sigma},0),
  \]
  where $\Norm(0,\mathbf{\Sigma})$ is the multivariate normal
  distribution with mean 0, and covariance matrix $\mathbf{\Sigma}$.
\end{thmab}
  
We will see throughout this work that the matrix $\mathbf{\Sigma}$ can
be explicitly computed in terms of a quadratic form built from a
\emph{discrete Green's function} associated to the exclusion process
(Theorem \ref{SpectralBound}). We will also see that this description
lends itself to natural spectral bounds on the entries of
$\mathbf{\Sigma}$. Such bounds are useful because, as conjectured by
D.\ Aldous and proven by \ocite{CLR}, the spectral gap for
the exclusion process is the same as the spectral gap for a delayed
random walker on the graph.

Our interest in $W(t)$ goes beyond simple curiosity as well. It is a
fact that the homotopy type of the space $\F_2(G)$ is determined by
combinatorial data that is far weaker than even the degree sequence
of $G$. In particular, no topological invariants of $\F_2(G)$ are
capable of recovering the tree $G$.  Nonetheless, the determination
of the winding statistic and its interaction with the topology of
$\F_2(G)$ lead us to conjecture the following:

\begin{conjab}\label{mainconjab}
  Let $G$ and $G'$ be two planar leaf-rooted trees that do not have
  vertices of degree 2, and write $\mathbf{\Sigma}_G$ and
  $\mathbf{\Sigma}_{G'}$ for the covariance matrices associated to the
  winding of the exclusion process on $G$ and $G'$, respectively. If
  $\mathbf{\Sigma}_G = \mathbf{\Sigma}_{G'}$, then $G$ and $G'$ are
  isomorphic as trees. 
\end{conjab}

In fact, we wish to know even if it possible to
recover the degree sequence of $G$ from its
associated covariance matrix.  It is the intention of the authors to
prove Conjecture \ref{mainconjab}, at least in the case of unitrivalent trees, in
future work. In this work, we provide evidence of Conjecture
\ref{mainconjab} through worked  examples as well as numerical
experiments (Section \ref{windEx}).  

The rest of the paper is organized as follows:  In Section
\ref{sec:BG} we provide background on configuration spaces and
Markov chains.  In Section \ref{sec:AH} we define the accumulated
homolgy of a random walk on a graph, provide some general facts and
work out specific cases.  In Section \ref{sec:Wind} we discuss the
winding of Random walks in $\F_n(G)$ and give some examples.

\section{Background} \label{sec:BG}

\subsection{Graph configuration spaces}\label{confBack}

In this section we outline some basic theory from the theory of graph configuration spaces. Because of the overall context of this work, we only treat the case of trees. One can find more general treatments of this material in \ocites{A,ADK,FS,R}.

\begin{definitions} \label{defn:gcs}
A \emph{graph} is a connected and finite one-dimensional CW
complex. We assume that all graphs are simple, i.e., they do not
contain self-loops or multi-edges. For a graph $G$, we write $V(G)$
for the set of its \emph{vertices}, or zero-cells, and
$E(G)$ for the set of its \emph{edges}, or one-cells. The
\emph{degree} $\deg(x)$ of a vertex $x$ of $G$ is defined to be
the number of edges containing it as an endpoint. The
\emph{boundary} of a cell $\sigma$ of $G$ is defined as 
\[
  \partial(\sigma) =
  \begin{cases}
    \{\sigma\} & \text{ if $\sigma$ is a vertex of $G$}\\
    \{\text{the endpoints of $\sigma$}\} & \text{ otherwise.}
  \end{cases}
\]
A \emph{tree} is a contractible graph.

Given a graph $G$, the \emph{$n$-stranded configuration space of $G$}
is the quotient topological space
\[
  \F_n(G) := \{(x_1,\ldots,x_n) \in G^n \mid x_i \neq x_j
  \text{ whenever $i \neq j$}\}/\Sn_n,
\]
where the symmetric group $\Sn_n$ acts by permuting indices.
\end{definitions}

\begin{remark}
  It is common in the literature to refer to $\F_n(G)$ as the
  \emph{unordered} $n$-stranded configuration space of $G$ to stress
  that the space is modulo the action of the symmetric group. In this
  work we will never deal with the space where this quotient is not
  performed, i.e., the \emph{ordered} configuration space.
\end{remark}

Much of the early literature on graph configuration spaces focused
on constructing combinatorial models for these spaces which could ease
computation. (See, for example, \ocites{A,G,Sw}.) In this work, we
will extensively use the model of Abrams. 

\begin{definition}
  For a graph $G$ and $n \geq 0$, the quotient of the product space
  $G^n/\Sn_n$ inherits an obvious cellular structure. Specifically, the
  cells of $G^n/\Sn$ are expressible as 
  \[
    \{\sigma_1 , \ldots ,\sigma_n\} \,,
  \]
  where for each $i$, the cell $\sigma_i$ is either a vertex of $G$ or an
  edge. The \emph{discretized $n$-stranded configuration space of
    $G$} is the subcomplex of $G^n/\Sn_n$ generated by cells of the
  form 
  \[
     \{\sigma_1 , \ldots ,\sigma_n\}
  \]
  where $\partial(\sigma_i) \cap \partial(\sigma_j) = \emptyset$. We
  denote this space by $D\F_n(G)$ 
\end{definition}

One immediately observes that $D\F_n(G)$ embeds as a
subspace of $\F_n(G)$, making it natural to ask whether there exists a
deformation retract of $\F_n(G)$ onto $D\F_n(G)$. Generally speaking
this cannot be the case as, for instance, $D\F_n(G)$ is empty whenever
the number of vertices of $G$ is smaller than $n$. However, concerns of this form turn
out to be the only obstructions to the existence of the deformation
retract. 

\begin{theorem}[Abrams, Theorem A]\label{cellularModel}
  Let $G$ be a tree, and assume that there are at least $n$ vertices
  on the path connecting any two vertices of $G$ of degree not equal
  to two. Then the natural inclusion $D\F_n(G) \hookrightarrow
  \F_n(G)$ is a homotopy equivalence. 
\end{theorem}

\begin{remark}
  Abrams' original theorem proves a similar result without the
  assumption that $G$ is a tree. In this work we will only need the
  above. 
\end{remark}

One should observe that while $D\F_n(G)$ is influenced by subdividing
edges of $G$, the same is not true about $\F_n(G)$. Therefore, the
primary assumption of Theorem \ref{cellularModel} about the lengths of
paths in $G$ is not particularly detracting. One should also observe
that this path length assumption will be satisfied for every tree when
$n = 2$. 

While Abrams' cellular model is the first step in simplifying homology
computations, it is unfortunately saddled with an over abundance of
cells of every dimension, making it difficult to do any kind of
computation when $n$ or $G$ is large. To simplify matters even
further, we must apply techniques from the so-called discrete Morse
theory of Forman \ycites{Fo,Fo2}. The following definitions were first
given in \ocite{FS}. 

\begin{definition}\label{criticalCells}
  Let $G$ be a tree, and assume that $G$ satisfies the path-length
  condition of Theorem \ref{cellularModel}. $1$-cells of $D\F_n(G)$
  can be encoded as sets 
  \[
    \{\sigma_1,\ldots,\sigma_n\} \,,
  \]
  where $\sigma_i$ is a cell of $G$, and there is precisely one index
  $i$ for which $\sigma_i$ is an edge of $G$.

  Fix a choice of embedding of $G$ into the plane, as well as a choice
  of vertex of degree one. We call this chosen vertex the \emph{root}
  of $G$. These two choices induce a well-ordering on the vertices of
  $G$ via a depth-first progression from the root. Given an edge
  $\sigma$ of $G$, we write $\tau(\sigma)$ to denote the endpoint of
  $\sigma$ with smaller label, and $\iota(\sigma)$ to denote the
  endpoint of $\sigma$ with larger index. If $\sigma$ is a vertex of
  $G$, which is not the root, we write $e(\sigma)$ for the unique edge
  for which $\iota(e(\sigma)) = \sigma$. 
  
   \emph{In the sequel, our discussion of the 1-cells of $D\F_n(G)$, we
    will always assume that $\sigma_1$ is an edge of $G$. We also
    assume that the above 1-cell is oriented from
    $\{\tau(\sigma_1),\ldots,\sigma_n\}$ to
    $\{\iota(\sigma_1),\ldots,\sigma_n\}$}. 
  
  Given a $1$-cell of $D\F_n(G)$, $c := \{\sigma_1,\ldots,\sigma_n\}$,
  we say that a vertex $\sigma_i \in c$ is \emph{blocked in $c$} if
  either $\sigma_i$ is the root of $G$, or
  $\{\sigma_1,\ldots,\sigma_{i-1},e(\sigma_i),\sigma_{i+1},\ldots,\sigma_n\}$
  is \emph{not} a two cell of $D\F_n(G)$. We also say that the edge
  $\sigma_1$ is \emph{order respecting} if for all vertices $\sigma_i
  \in c$ such that $\tau(e(\sigma_i)) = \tau(\sigma_1)$, 
 the label of $\iota(\sigma_1)$ is smaller than that of $\sigma_i$. We
  say that the $1$-cell $c$ is \emph{critical} if $\sigma_1$ is
  \emph{not} order respecting, and all vertices of $c$ are blocked. We
  instead say that the $1$-cell $c$ is \emph{collapsible} if
  $\sigma_1$ is order respecting and every unblocked vertex has a label
  which is strictly bigger than $\iota(\sigma_1)$.
\end{definition}

When $G$ is the $H$-shaped graph with a single
subdivision on each edge, examples of critical and collapsible 
$1$-cells of $D\F_3(G)$ are given in Figure \ref{critCellExample}. Note
that implicit in this image is our choice of root vertex, which we
have labeled by 0. 

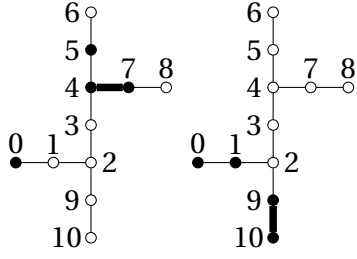
\begin{figure}
\begin{tikzpicture}
\tikzstyle{every node}=[draw,circle,fill=white,minimum size=4pt,
                            inner sep=0pt]
    \draw (0,0) node (0) [fill = black, label=above:$0$] {}
		      --(.5,0) node (1) [label=above:$1$]{}
					--(1,0) node (2) [label=right:$2$]{}
					--(1,.5) node (3) [label=left:$3$]{}
					--(1,1) node (4) [fill = black, label=left:$4$]{}
					--(1,1.5) node (5) [fill = black, label=left:$5$]{}
					--(1,2) node (6) [label=left:$6$]{}
    			(4) -- (1.5,1) node (7) [fill = black, label=above:$7$]{}
					--(2,1) node (8) [label=above:$8$]{}
					(2) -- (1,-.5) node (9) [label=left:$9$]{}
					--(1,-1) node (10) [label=left: $10$]{};
		 \draw [line width=1mm] (4) -- (7);
\end{tikzpicture}
\begin{tikzpicture}
\tikzstyle{every node}=[draw,circle,fill=white,minimum size=4pt,
                            inner sep=0pt]
    \draw (0,0) node (0) [fill = black, label=above:$0$] {}
		      --(.5,0) node (1) [fill = black, label=above:$1$]{}
					--(1,0) node (2) [label=right:$2$]{}
					--(1,.5) node (3) [label=left:$3$]{}
					--(1,1) node (4) [label=left:$4$]{}
					--(1,1.5) node (5) [label=left:$5$]{}
					--(1,2) node (6) [label=left:$6$]{}
    			(4) -- (1.5,1) node (7) [label=above:$7$]{}
					--(2,1) node (8) [label=above:$8$]{}
					(2) -- (1,-.5) node (9) [fill = black, label=left:$9$]{}
					--(1,-1) node (10) [fill = black, label=left: $10$]{};
		 \draw [line width=1mm] (9) -- (10);
\end{tikzpicture}
\caption{The critical $1$-cell $\{v_0,v_5,e_{4,7}\}$, and the
  collapsible $1$-cell $\{v_0,v_1,e_{9,10}\}$}\label{critCellExample} 
\end{figure}

The following theorem follows from work of Farley and Sabalka
\ycite{FS}) as well as follow-up work of Farley \ycite{Fa}).

\begin{theorem}[\ocites{FS,Fa}]\label{homologyBasis}
  If $G$ is a tree satisfying the path-length
  condition of Theorem \ref{cellularModel}, then
  \begin{enumerate}
  \item the subcomplex of $D\F_n(G)$ generated by the collapsible one
    cells forms a spanning tree of the one-skeleton of $D\F_n(G)$; 
  \item the complex $D\F_n(G)$ is homotopy equivalent to a cellular
    complex $\mathcal{M}$ whose unique $0$-cell corresponds to the
    unique $0$-cell of $D\F_n(G)$ whose every vertex is blocked, and
    whose $1$-cells correspond to the critical $1$-cells of
    $D\F_n(G)$. Moreover, the cellular differential 
    \[
      \Z\mathcal{M}^{(2)} \rightarrow \Z\mathcal{M}^{(1)}
    \]
    from the two cells of $\mathcal{M}$ to the $1$-cells is the zero
    map. In particular, the first homology group $H_1(\F_n(G))$ is
    free, with a basis in bijection with the critical $1$-cells of
    $D\F_n(G)$. 
\end{enumerate}
\end{theorem}

\begin{remark}
  The homotopy equivalence of the prior theorem is actually what is
  known as a \emph{cellular collapse}. In particular, there is a
  formal sense in which one can say that the $1$-cells of $\mathcal{M}$
  are the same as the critical $1$-cells of $D\F_n(G)$. This will be
  used implicitly throughout the work. 
\end{remark}

The spanning tree of the one-skeleton of $D\F_n(G)$ given by the
collapsible $1$-cells has a nice description, which we now give in the
case $n = 2$. The general case is similar, though a bit more
cumbersome to state. 

Given two pairs of distinct vertices $\{v_1,w_1\}, \{v_2,w_2\}$, the
path connecting them in our spanning tree is described as
follows. Between $v_1$ and $w_1$, select the vertex with the smaller
label. Say that this vertex is $v_1$. Then one keeps $w_1$ fixed,
while allowing for $v_1$ to flow to the root. Once $v_1$ arrives at
the root, $w_1$ is then moved to the bigger of $v_2$ and
$w_2$. Finally, $v_1$ is moved from the root to the smaller of $v_2$
and $w_2$.

While Theorem \ref{homologyBasis} provides a complete description of
the first homology of $D\F_n(G)$, and therefore also $\F_n(G)$, for
our intended applications we need to be a bit more detailed in our
understanding on how the aforementioned homotopy equivalences collapse
the 2-cells of $D\F_n(G)$. We begin with the following definition. 

\begin{definition}
  Let $c = \{\sigma_1,\ldots,\sigma_n\}$ denote a critical 1-cell of
  $D\F_n(G)$. Then the edge $\sigma_1$ has its smaller endpoint
  $\tau(\sigma_1)$ on a vertex of $G$ of degree $\geq 3$. For this
  definition only, write $G'$ as the collection of connected
  components of $G - \tau(\sigma_1)$.\\ 

  Let $c' = \{\tau_1,\ldots,\tau_n\}$ be a 1-cell of $D\F_n(G)$ which
  is not collapsible.  We say $c'$ \emph{lies on top of $c$} if, 
  \begin{enumerate}
  \item the edge $\tau_1$ is equal to the edge $\sigma_1$ and,
  \item for every component in $G'$, the number of vertices $c$'
    contains in this component is the same as the number of vertices
    $c$ contains in this component. 
\end{enumerate}
\end{definition}

The following theorem can be proven using the second part of Theorem
\ref{homologyBasis}, as well as Proposition 2.2 and Lemma 5.1 of
\ocite{FS}. 

\begin{theorem}\label{howCellsWork}
  Let $G$ be a tree, and assume that $G$ satisfies the path-length
  condition of Theorem \ref{cellularModel}. With respect to the basis
  of $H_1(D\F_n(G)^{(1)})$ indexed by the non-collapsible 1-cells of
  $D\F_n(G)$, the composition 
  \[
    H_1(D\F_n(G)^{(1)})
    \rightarrow H_1(D\F_n(G)) \stackrel{\cong}{\rightarrow} H_1(\mathcal{M})
  \]
  is defined on points by
  \[
    (\alpha_c)_{c \textrm{ not collapsible}} \mapsto (\sum_{c \text{
        lies on top of $c'$}}\alpha_c)_{c' \text{ critical}} 
  \]
\end{theorem}

\subsection{Markov chains}

In this section, we give a brief overview of relevant topics related
to Markov chains. The example of random walks on graphs will be used
both as motivation and as a way to ground the material. All of what
follows can be found in any standard text on the subject. (See, for
example, \ocite{LPW}.)

In the following $(X_t)_{t = 0}^\infty$ will be a Markov chain on the
state space $\X$ with transition matrix $P$, so that, for all $t > 0$
and $x_0,\ldots,x_t \in \X$,
\[
  \P(X_t = x_t \mid X_0 = x_0, \ldots, X_{t-1} = x_t) = P(x_{t-1},x_t)
  \,.
\]
The chain is \text{irreducible} if for all $v,w \in \X$, there
exists $r = r(v,w)$ such that $P^r(v,w) > 0$.   If $P$ is irreducible,
then there exists a unique stationary distribution on $\X$, i.e.\
a probability row vector satisfying $\pi = \pi P$.

A \emph{reversible} Markov chain satisfies the \emph{detailed balance}
equations
\[
  \pi(x)P(x,y) = \pi(y)P(y,x) \quad \text{ for all } x,y \in {\mathcal
    X} \,.
\]
We say that $P$ is \emph{lazy} if for any state $x \in \X$, $P(x,x) \geq \frac{1}{2}$.


\begin{remark}
  If $P$ is lazy, then all the eigenvalues of $P$ are non-negative.
  Assuming the eigenvalues of $P$ are non-negative allows for simple
  descriptions of certain spectral bounds. (For example, the proof of Theorem
  \ref{SpectralBound}.) 
\end{remark} 

A \emph{statistic} on the state space of $\X$ is a function $f:\X \to
\R$.   The expectation of a statistic $f$ with respect to $\pi$ will
be denoted
\[
  E_\pi(f) = \sum_{x \in \X} f(x) \pi(x) \,.
\]
and the variance is
\[
  \Var_\pi(f) = E_\pi[( f - E_\pi(f))^2] \,.
\]

Letting
\[
  \widehat{f}_n:= \sum_{t = 0}^n \frac{f(X_t)}{\sqrt{n}} \,,
\]
we record below the central limit theorem of $\hat{f}_n$:

\begin{theorem}[\cite{Bolt}*{Theorem 1}]\label{MarkovCLT}
  Let $P$ be a connected aperiodic Markov chain on a finite state
  space $\X$, and let $f$ be a statistic on $\X$. Then, 
  \[
    \widehat{f}_n - \sqrt{n}E_\pi(f)
    \stackrel{\mathcal{D}}{\rightarrow} N(0,\sigma^2),
  \]
  where $N(0,\sigma^2)$ is the centered normal distribution with
  variance $\sigma^2$, and 
  \begin{align}
    \sigma^2 = \Var_\pi(f) + 2\sum_{t \geq 1}\Cov_\pi(f(X_0),f(X_t)).
    \label{variance}
  \end{align}
  Moreover, writing $F_n(t)$ for the distribution function of
  $\widehat{f}_n - \sqrt{n} E_\pi(f)$, one has 
  \[
    \|F_n - N(0,\sigma^2)\|_\infty = O(n^{-\frac{1}{2}}).
  \]
\end{theorem}

\begin{remark}
  The first half of the above theorem, which states
  a central limit theorem for Markov chains, is classical and can be
  found in a variety of textbooks and surveys;  see, for example
  \ocite{Jones}.   \ocite{Bolt} proves the second half of the
  above theorem, which implies that the convergence rate of this central
  limit theorem is on the order of $n^{-\frac{1}{2}}$. 
\end{remark}

The \emph{Cramer-Wold} device is a method of obtaining a multivariate
central limit theorem.  Suppose $\bvec{X}_t$ and $\bvec{W}$ are $d$-dimensional random
variables.  If
\[
  \theta \cdot \bvec{X}_t \Rightarrow \theta \cdot \bvec{W}
\]
for all $\theta$, then $\bvec{X}_t \Rightarrow \bvec{W}$.
This follows directly from the equivalence of weak convergence to
convergence of Fourier transforms.  See, for example, \ocite{Bill}. For our purposes, the Cramer-Wold device will be used to conclude certain multi-variate statistics converge to a normal distribution, by showing that every projection of the statistic does so.


\section{Accumulated homology of a random walk on a graph}
\label{sec:AH}

\subsection{The edge-walk}

In this section we outline a useful method for recording edge traversals during a random walk on some graph.\\

\emph{For the remainder of this section, we fix a graph $G$, as well as a model of a random walk on $G$, $(X_t,P)$. We will also assume that the Markov chain is lazy.}

\begin{definition}
  We write $\mathcal{E}(G)$ for the collection of \emph{directed edges
    of $G$ with respect to $X_t$}, 
  \[
    \mathcal{E}(G)
    = \{(x,y) \mid P(x,y) \neq 0\} \subseteq V(G) \times V(G).
  \]
  The associated \emph{edge-walk} of $(X_t,P)$ is the Markov chain
  $(Y_t,P_E)$ on the state space $\E(G)$ given by, 
  \[
    Y_t = (X_t,X_{t+1}), \hspace{1cm} P_E((x,y),(z,w)) =
    \begin{cases} P(z,w) &\text{ if $z = y$}\\
      0 &\text{ otherwise.}
    \end{cases}
  \]
\end{definition}

For much of this section, we will be proving a variety of elementary properties of the edge-walk associated to $(X_t,P)$. In particular, we discuss how certain properties of $(X_t,P)$ relate to those of $(Y_t,P_E)$.\\

We begin with a computation of the stationary distribution of $P_E$. Note that $(Y_t,P)$ is connected by our assumptions on $(X_t,P)$, and therefore has a unique stationary distribution.

\begin{proposition}\label{stationaryFormula}
Let $\pi$ denote the stationary distribution of $(X_t,P)$, and $\pi_E$ the stationary distribution of $(Y_t,P_E)$. Then, for any $(x,y) \in \E(G)$,
\[
\pi_E(x,y) = \pi(x)P(x,y)
\]
\end{proposition}

The proof is simply checking that the detailed balance equations hold;
we omit the details.

It follows from Proposition \ref{stationaryFormula} that the edge walk
$(Y_t,P_E)$ is not reversible. Despite this, one immediate and
ultimately very useful consequence of this description of $\pi_E$ is
the following stand-in for reversibility. 

\begin{corollary}
  Let $\pi_E$ denote the stationary distribution of $(Y_t,P_E)$. Then,
  \begin{align}
    \pi_E(x,y) = \pi_E(y,x)\label{almostRev}
  \end{align}
\end{corollary}

\begin{proof}
  Follows immediately from Proposition \ref{stationaryFormula} and the
  fact that $(X_t,P)$ is reversible. 
\end{proof}

\subsection{Accumulated homology: The general case}

In this section, we provide the theoretical framework for studying what we term \emph{accumulated homology} of a random walk on a graph. In the section that follows, we will provide worked examples illustrating the method.

For the remainder of this section, we fix a graph $G$ as well as a
model of a random walk on $G$, $(X_t,P)$, which we assume to be
aperiodic. We also follow the notation from the previous section and
write $(Y_t,P_E)$ for the associated edge walk of $(X_t,P)$. Finally,
we fix now for all time a spanning tree $T_G$ for $G$. The edges in the compliment of $T_G$ will be written as $e_1,...,e_g$. The endpoints of $e_i$ will be written as $x_i$ and $y_i$, and we orient this edge as $(x_i,y_i)$. We write $e_i^{-}$ for the reversed edge $(y_i,x_i)$.

\begin{definition}
Let $\rho = (\rho_0,\rho_1,\ldots,\rho_r)$ be a path of length $r$ in $G$.h e \emph{closure} of $\rho$, denoted $\overline{\rho}$, is defined to be the loop in $G$ obtained from $\rho$ by concatenating the unique path in $T_G$ from $v_r$ to $v_0$. The \emph{accumulated homology of $\rho$ (with respect to $T_G$)} is the tuple
\[
H_1^{T_G}(\rho) := (h_1,\ldots,h_g) \in \Z^g,
\]
where,
\[
h_i = |\{j \mid \rho_j = x_i, \rho_{j+1} = y_i\}| - |\{j \mid \rho_j = y_i, \rho_{j+1} = x_i\}|.
\]
\end{definition}

\begin{remark}
The accumulated homology of $\rho$ is precisely the class in $H_1^{T_G}(G)$ represented by $\overline{\rho}$ with respect to the basis given by the oriented edges $e_i$. One may think of the closure of $\rho$ as a loop obtained from $\rho$ in such a way as to not "create more homology" than had already been accumulated by $\rho$.

Let $T'_G$ be a different choice of spanning tree, with (ordered and oriented) complementary edges $(e'_1,\ldots,e'_g)$. Writing $e_i$ for the basis vector of $H_1(G)$ corresponding to the oriented edge $e_i$, then the change of basis from $\{e_i\}$ to $\{e'_j\}$ is obtained in the following way. The unique path from $y_i$ to $x_i$ through $T_G$ will traverse a variety of the edges $e'_j$. Then the assignment
\[
e_i \mapsto \sum_{j} \alpha_j e'_j,
\]
where $\alpha_j$ records the (net) number of traversals of $e'_j$, defines the necessary change of basis. By the previous remark, this gives one a simple means to relate $H_1^{T_G}(\rho)$ to $H_1^{T'_G}(\rho)$. For this reason, we usually suppress the spanning tree $T_G$ in the notation for the accumulated homology. 
\end{remark}

Given $t \geq 0$, the collection of vertices
\[
(X_0, X_1, \ldots, X_t)
\]
defines a path in $G$, which we denote $\rho_X(t)$. The primary object of study in this section is the random variable
\[
H_1(\rho_X(t))
\]
More specifically, our interest will be in understanding the limiting distribution
of $H_1(\rho_X(t))$.

\begin{theorem}\label{ahclt}
There exists a $g \times g$ matrix $\mathbf{\Sigma}$ such that
\[
\lim_{t \to \infty} \frac{H_1(\rho_X(t))}{\sqrt{t}}
\stackrel{\mathcal{D}}{\rightarrow} N(0,\mathbf{\Sigma})
\]
where $\Norm(0,\mathbf{\Sigma})$ is the multivariate normal distribution with covariance matrix $\mathbf{\Sigma}$.
\end{theorem}

\begin{proof}
Write $\epsilon_i$ for the standard $i$-th canonical basis vector of $\R^g$. For this proof only, we define a multi-variate statistic $f:\R \E(G) \rightarrow \R^g$ by setting
\[
f(x,y) = \begin{cases} \epsilon_i &\text{$(x,y) = (x_i,y_i)$}\\ -\epsilon_i &\text{$(x,y) = (y_i,x_i)$}\\ 0 &\text{ otherwise.}\end{cases}
\]
Then we may write,
\[
 H_1(\rho_X(t)) = \sum_{n = 0}^t f(Y_n).
\]
In particular, the multivariate central limit theorem implies
\[
\sqrt{t}(H_1(\rho_X(t))/t - E_{\pi_E}(f)) \stackrel{\mathcal{D}}{\rightarrow} N(0,\mathbf{\Sigma})
\]
for some covariance matrix $\Sigma$ depending only on $f$, and where $\pi_E$ is the stationary distribution of $(Y_t,P_E)$. By (\ref{almostRev}), it is easily seen that $E_{\pi_E}(f) = 0$, whence we obtain
\[
H_1(\rho_X(t))/\sqrt{t} \stackrel{\mathcal{D}}{\rightarrow}
\Norm(0,\mathbf{\Sigma}) \,.
\]
\end{proof}

\begin{remark}
  As a functional of an ergodic Markov chain, it is obvious that $W$
  should obey a central limit theorem.   Our purpose of recording this
  result is that the matrix $\mathbf{\Sigma}$ should encode topological
  information in our primary application to configuration spaces. Moreover, we will see later
  that entries of $\mathbf{\Sigma}$ can be bounded in terms of spectral properties of the graph.
\end{remark}  

We dedicate the remainder of this section to providing bounds on the
entries of the matrix $\mathbf{\Sigma}$. These bounds will be given in
terms of the spectral gap of the matrix $P$. We will also compute a
closed form for $\mathbf{\Sigma}$ in terms of the so-called discrete
Green's functions associated to the walk. (See, e.g., \cite{CY} for a reference.) To aid us in this computation, we state the following convenient notation.

\begin{definition}
The vector space $\R \E^{\vee} = \Hom_\R(\R \E, \R)$ carries the structure of an inner-product space via the assignment,
\[
\langle f, g \rangle = \sum_{e \in \E} \pi_E(e)f(e)g(e)
\]
For each $1 \leq i \leq g$, we write $\mathbf{1}_i:\E \rightarrow \R$ for the statistic
\[
\mathbf{1}_i(x,y) = \begin{cases} 1 &\text{$(x,y) = (x_i,y_i)$}\\ -1 &\text{$(x,y) = (y_i,x_i)$}\\ 0 &\text{ otherwise.}\end{cases}
\]
In particular, we may write
\[
H_1(\rho_X(t)) = \sum_{n = 0}^t (\mathbf{1}_1(Y_n),\ldots,\mathbf{1}_g(Y_n)).
\]

We similarly define the $\pi$-normalized inner product on $\R V(G)^{\vee}$, as well as the statistics
\[
f_{i}(z) = \begin{cases} P(x_i,y_i) &\text{ if $z = x_i$,}\\ -P(y_i,x_i) &\text{ if $z = y_i$}\\ 0 &\text{ otherwise.}\end{cases}
\]

If $f:V(G) \rightarrow \R$ is a statistic, we will write
\[
\|f\|^2_{\pi} := \langle f,f \rangle
\]
\end{definition}

We begin with two short technical lemmas. The first illustrates the relationship between $P_E^t$ and $P^t$. Its proof is straight forward.

\begin{lemma}\label{reduceToVertex}
Let $e = (x,y),e' = (x',y')$ be two elements of $\E$. Then for any $t \geq 1$
\[
P_E^t(e,e') = P^{t-1}(y,x') \cdot P(x',y')
\]
\end{lemma}

Our second lemma gives us a convenient linear algebraic interpretation for certain covariances which will play a key role in Theorem \ref{SpectralBound}.

\begin{lemma}\label{covinner}
Let $\alpha = (\alpha_1,\ldots,\alpha_g)$ be some vector in $\R^g$. Then for any $t \geq 1$,
\begin{align*}
\Cov_{\pi_E}\Bigl(\sum_i \alpha_i \mathbf{1}_i(Y_0),\sum_i \alpha_i
  \mathbf{1}_i(Y_t) \Bigr)
  & = \Bigl\langle \sum_i \alpha_i \mathbf{1}_i,  P_E^t\cdot \sum_i
    \alpha_i\mathbf{1}_i \Bigr\rangle \\
  & = -\Bigl\langle \sum_i \alpha_i f_i, P^{t-1} \cdot \Bigl(\sum_j \alpha_j f_j\Bigr) \Bigr\rangle
\end{align*}
\end{lemma}

\begin{proof}
We may write,
\begin{align*}
\Cov_\pi\Bigl(\sum_i \alpha_i \mathbf{1}_i(Y_0),\sum_j \alpha_j
  \mathbf{1}_j(Y_t)\Bigr)
  &= \sum_{i,j}\alpha_i\alpha_j\Cov_\pi(\mathbf{1}_i(Y_0), \mathbf{1}_j(Y_t)).
\end{align*}

Using that $\mathbf{1}_i$ is mean zero and (\ref{almostRev}), we have:
\begin{align*}
\Cov_\pi(\mathbf{1}_i(Y_0), \mathbf{1}_j(Y_t)) &= E_{\pi_E}(\mathbf{1}_i(Y_0) \cdot \mathbf{1}_j(Y_t))\\
											 &=
                                                                                           \pi_E(e_i)P_E^t(e_i,e_j)
                                                                                           +
                                                                                           \pi_E(e_i^{-})
                                                                                           P_E^t(e_i^{-},e_j^{-})
  \\
  & \qquad - (\pi_E(e_i)P_E^t(e_i ,e_j^{-}) + \pi_E(e_i^{-}) P_E^t(e_i^{-},e_j))\\
											 &=
                                                                                           \pi_E(e_i)(P_E^t(e_i,e_j)-
                                                                                           P_E^t(e_i,e_j^{-}))
  \\
  & \qquad - \pi_E(e_i^{-})(P_E^t(e_i^{-},e_j) - P_E^t(e_i^{-},e_j^{-}))\\
											 &= \langle \mathbf{1}_i,  P_E^t \cdot \mathbf{1}_j \rangle.
\end{align*}

Putting the previous two computations together we have,
\[
\Cov_\pi\Bigl(\sum_i \alpha_i \mathbf{1}_i(Y_0),\sum_j \alpha_j \mathbf{1}_j(Y_t)\Bigr) = \sum_{i,j}
\alpha_i\alpha_j\langle \mathbf{1}_i, P_E^t \cdot \mathbf{1}_j \rangle = \Bigl\langle \sum_i \alpha_i \mathbf{1}_i, P_E^t \cdot \sum_j \alpha_j\mathbf{1}_j \Bigr\rangle,
\]
proving the first equality. To finish the proof, we note by Lemmas \ref{stationaryFormula} and \ref{reduceToVertex}, as well as (\ref{almostRev}),\\
\begin{align*}
\langle \mathbf{1}_i,  P_E^t \cdot \mathbf{1}_j \rangle = &\pi(y_i)P(y_i,x_i)(P(x_j,y_j)P^{t-1}(y_i,x_j)- P(y_j,x_j)P^{t-1}(y_i,y_j))\\
&- \pi(x_i)P(x_i,y_i)(P(x_j,y_j)P^{t-1}(x_i,x_j) - P(y_j,x_j)P^{t-1}(x_i,y_j)),
\end{align*}
which is immediately seen to be equal to $-\langle f_i, P^{t-1} \cdot f_j \rangle$.
\end{proof}

With these two lemmas in hand we are ready to begin producing our desired bounds. We begin with the following.

\begin{theorem}\label{SpectralBound}
Let $\mathbf{\Sigma}$ be as in the statement of Theorem \ref{ahclt}, let $\alpha = (\alpha_1,\ldots,\alpha_g)$ be some vector in $\R^g$, and set
\[
\sigma^2(\alpha) := \alpha^\dagger\mathbf{\Sigma}\alpha 
\]
Then,
\[
\sigma^2(\alpha) = 2\sum_{i}\alpha_i^2\pi_E(e_i) - 2\Bigl\langle (I-P)^{-1}\Bigl(\sum_i \alpha_i f_i\Bigr), \sum_j \alpha_j f_j \Bigr\rangle.
\]
In particular, one has
\[
2 \sum_i \alpha_i^2\pi_E(e_i) - \frac{2 \|\sum_i \alpha_i
  f_i\|^2_\pi}{\delta_{X_t,P}} \leq \sigma^2(\alpha) \leq
2\sum_{i}\alpha_i^2\pi_E(e_i) - \frac{2 \|\sum_i \alpha_i
  f_i\|^2_\pi}{1-\gamma_{{\rm min}}},
\]
where $\delta_{X_t,P}$ is the spectral gap of $(X_t,P)$ and $\gamma_{min}$ is the smallest eigenvalue of $P$.
\end{theorem}

\begin{remark}
Before we commence with the proof, it is important that we note what is meant by $(I-P)^{-1}$. Indeed, because $P$ is stochastic, it is a fact that $(I-P)$ is not invertible. That being said, however, it will be invertible on the orthogonal complement of the (unique) eigenvector of $P$ corresponding to the eigenvalue 1. The statistic $\sum_i \alpha_i f_i$ always has mean zero, by virtue of the fact that this is the case for all of the $f_i$, and therefore the expression $(I-P)^{-1}(\sum_i \alpha_i f_i)$ makes sense.\\

To be completely precise, $(I-P)^{-1}$ is the unique matrix satisfying
\[
(I-P)^{-1}(I-P) = (I-P)(I-P)^{-1} = I-P_0,
\]
where $P_0$ is the projection onto the vector $(\sqrt{\pi(x)})_{x \in V(G)}$.
\end{remark}

\begin{proof}
If we write $\alpha = (\alpha_1,\ldots,\alpha_g)$ for a vector in $\R^g$, then (\ref{variance}) implies we must compute
\[
\sigma^2(\alpha) = \Var_{\pi_E}\Bigl(\sum_i \alpha_i \mathbf{1}_i\Bigr) + 2\sum_{t \geq 1}\Cov_{\pi_E}\Bigl(\sum_i \alpha_i \mathbf{1}_i(Y_0),\sum_i \alpha_i \mathbf{1}_i(Y_t)\Bigr).
\]
We first compute $\Var_{\pi_E}\Bigl(\sum_i \alpha_i \mathbf{1}_i\Bigr)$. To start:
\begin{align*}
\Var_{\pi_E}\Bigl(\sum_i \alpha_i \mathbf{1}_i\Bigr) = E_{\pi_E}\Bigl(\Bigl(\sum_i \alpha_i \mathbf{1}_i\Bigr)^2\Bigr) - E_{\pi_E}\Bigl(\sum_i \alpha_i \mathbf{1}_i\Bigr)^2 = E_{\pi_E}\Bigl(\Bigl(\sum_i \alpha_i \mathbf{1}_i\Bigr)^2\Bigr).
\end{align*}
On the other hand, using (\ref{almostRev}) as well as the fact that $\mathbf{1}_i(Y_0)\mathbf{1}_j(Y_0) = 0$ whenever $i \neq j$,
\begin{multline*}
  E_{\pi_E}\Bigl[\bigl(\sum_i \alpha_i \mathbf{1}_i)^2\bigr)\Bigr]
  = \sum_i \alpha_i^2 E_{\pi_E}[\mathbf{1}_i^2]
  = \sum_i \alpha_i^2 (\pi_E(e_i) + \pi_E(e_i^{-})) \\
  = 2\sum_i \alpha_i^2 \pi_E(e_i) \,.
\end{multline*}

We next turn our attention to the covariance terms $\Cov_\pi(\sum_i \alpha_i \mathbf{1}_i(Y_0),\sum_j \alpha_j \mathbf{1}_j(Y_t))$. Lemmas \ref{covinner} and \ref{reduceToVertex} imply
\[
\Cov_\pi\Bigl(\sum_i \alpha_i \mathbf{1}_i(Y_0),\sum_j \alpha_j \mathbf{1}_j(Y_t)\Bigr) = -\Bigl\langle \sum_i \alpha_i f_i, P^{t-1} \cdot \sum_j \alpha_jf_j \Bigr\rangle.
\]
Therefore,
\begin{align*}
\sum_{t \geq 1} \Cov_\pi\Bigl(\sum_i \alpha_i \mathbf{1}_i(Y_0),\sum_j
  \alpha_j \mathbf{1}_j(Y_t) \Bigr)
  & = \sum_{t \geq 0} -\Bigl\langle \sum_i \alpha_i f_i, P^{t} \cdot
    \sum_j \alpha_jf_j \Bigr\rangle\\
  & = -\Bigl\langle \sum_i \alpha_i f_i, \Bigl(\sum_{t \geq 0}P^{t}\Bigr) \cdot \sum_j \alpha_jf_j \Bigr\rangle\\
 & = -\Bigl\langle \sum_i \alpha_i f_i, (I-P)^{-1} \cdot \sum_j \alpha_jf_j \Bigr\rangle.
\end{align*}

This completes the desired computation. Moving on to our claimed bounds, elementary linear algebra, as well as the assumption that the Markov chain $(X_t,P)$ is lazy and therefore has non-negative eigenvalues,  tells us that the inner product $\Bigl\langle \sum_i \alpha_i f_i, P^{t-1} \cdot \sum_j \alpha_jf_j \Bigr\rangle$ can be bounded as
\[
\Bigl\|\sum_i \alpha_i f_i\Bigr\|^2_{\pi} \min\{\gamma_P\}^{t-1} \leq \Bigl\langle \sum_i \alpha_i f_i, P^{t-1} \cdot \sum_j \alpha_jf_j \Bigr\rangle \leq \Bigl\|\sum_i \alpha_i f_i\Bigr\|^2_{\pi} \max\{\gamma_P\}^{t-1},
\]
where the min and max are over the eigenvalues of $P$ not equal to 1. Therefore,
\[
2\sum_{t\geq 1} \Cov_\pi\Bigl(\sum_i \alpha_i \mathbf{1}_i(Y_0),\sum_j \alpha_i \mathbf{1}_j(Y_t)\Bigr) \geq -2 \Bigl\|\sum_i \alpha_i f_i\Bigr\|^2_{\pi} \sum_{t \geq 0} \max\{\gamma_P\}^{t} = \frac{-2 \Bigl\|\sum_i \alpha_i f_i\Bigr\|^2_{\pi}}{\delta_{X_t,P}}
\]
A similar computation also yields upper bounds. Combining this with the previously computed variance term we conclude our desired bounds.
\end{proof}

By taking $\alpha$ to be the elementary basis vector in direction $i$, the above implies that,
\[
\mathbf{\Sigma}(i,i) = 2\pi_E(e_i) - 2\langle (I-P)^{-1}f_i, f_i
\rangle
\]
\begin{multline*}
2\pi_E(e_i) - \frac{2\pi_E(e_i) (P(x_i,y_i) +
  P(y_i,x_i))}{\delta_{X_t,P}} \\
\leq \mathbf{\Sigma}(i,i) 
\leq 2\pi_E(e_i) - \frac{2\pi_E(e_i) (P(x_i,y_i) +
  P(y_i,x_i))}{1-\gamma_{{\rm min}}}
\end{multline*}
We will use this to obtain bounds on the off-diagonal terms of the covariance matrix $\mathbf{\Sigma}$ in what follows.

\begin{theorem}\label{offdiagonalbounds}
Let $\mathbf{\Sigma}$ be as in the statement of Theorem \ref{ahclt}, and fix $1 \leq i < j\leq g$. Then,
\begin{align}
\mathbf{\Sigma}(i,j) &= 2 \langle (I-P)^{-1}f_i,f_j \rangle\\
\mathbf{\Sigma}(i,j) &\geq \frac{\pi_E(e_i) (P(x_i,y_i) + P(y_i,x_i))
                       + \pi_E(e_j) (P(x_j,y_j) + P(y_j,x_j))}{1 -
                       \gamma_{min}}\\
  & \quad - \frac{\|f_i + f_j\|^2_\pi}{\delta_{X_t,P}} \nonumber \\
\mathbf{\Sigma}(i,j) &\leq \frac{\pi_E(e_i) (P(x_i,y_i) + P(y_i,x_i))
                       + \pi_E(e_j) (P(x_j,y_j) +
                       P(y_j,x_j))}{\delta_{X_t,P}} \\
  & \quad - \frac{\|f_i + f_j\|^2_\pi}{1 - \gamma_{ {\rm min}}} \nonumber
\end{align}
\end{theorem}

\begin{proof}
Write $\alpha$ for the vector in $\R^g$, with $\alpha_i = \alpha_j = 1$ and $\alpha_k = 0$ otherwise. We have,
\begin{align*}
\mathbf{\Sigma}(i,i)+ \mathbf{\Sigma}(j,j) + 2\mathbf{\Sigma}(i,j) & =
                                                                     \alpha^\dagger
                                                                     \mathbf{\Sigma}
                                                                     \alpha\\
                                                                   & = 2(\pi_E(e_i) + \pi_E(e_j)) - 2\langle (I-P)^{-1}(f_i + f_j),f_i +f_j)\rangle
\end{align*}
Comparing the above to the explicit computation of $\mathbf{\Sigma}(i,i)$ given by Theorem \ref{SpectralBound}, we obtain the desired formula for $\mathbf{\Sigma}(i,j)$. On the other hand, using the equality $\mathbf{\Sigma}(i,i) + \mathbf{\Sigma}(j,j) + 2\mathbf{\Sigma}(i,j) = \alpha^\dagger \mathbf{\Sigma} \alpha$, along with the upper bound for $\alpha^\dagger \mathbf{\Sigma} \alpha$ and lower bound for $\mathbf{\Sigma}(i,i)$ given in Theorem \ref{SpectralBound}, we recover the desired upper bound. Our lower bound is computed similarly.
\end{proof}

\begin{remark}\label{conjrmk}
It is notable that the norm term $\|f_i + f_j, f_i + f_j\|_{\pi}^2$ behaves differently depending on whether the edges $e_i$ and $e_j$ are adjacent or not. For instance, one can deduce using similar reasoning to the proof of Theorem \ref{offdiagonalbounds}, that
\begin{align*}
|\mathbf{\Sigma}(i,j)| & \leq \left(\pi_E(e_i) (P(x_i,y_i) + P(y_i,x_i))
  + \pi_E(e_j) (P(x_j,y_j) + P(y_j,x_j)) \right)\\
& \qquad \times \left( \frac{1}{\delta_{X_t,P}} - \frac{1}{1 - \gamma_{min}} \right)
\end{align*}
whenever $e_i$ and $e_j$ are non-adjacent. This follows from the fact that adjacency determines whether the cross term $\langle f_i, f_j \rangle$ is zero or not.
\end{remark}

\subsection{Accumulated homology: examples}

In this section we consider a worked examples of accumulated winding on certain graphs. In particular, we will compute the covariance matrix $\mathbf{\Sigma}^{(n)}$ for the accumulated homology of the simple (lazy) random walk on the complete graph $K_n$. To be clear, this process proceeds as follows: at each step a coin is flipped to decide whether movement will be attempted. Assuming this first test passes, one then chooses an edge uniformly at random. If the edge happens to be adjacent to the current position, then the current position changes to the other endpoint of the edge. Otherwise the process holds.

The following linear algebra lemma is completely standard.

\begin{lemma} \label{twoconstantmatrix}
Let $M$ denote the $n \times n$ matrix whose diagonal terms are some constant $c_1$, and whose off diagonal terms are some constant $c_2$. Then,
\[
\det(M) = (c_1-c_2)^{n-1}(c_2(n-1)+c_1)
\]

In particular, if $P$ denotes the transition matrix of the simple random walk on the complete graph $K_n$, then the distinct eigenvalues of $P$ are 1, appearing with multiplicity 1, and $1 - \frac{n}{2\binom{n}{2}}$ appearing with multiplicity $n-1$.
\end{lemma}

One important consequence of Lemma \ref{twoconstantmatrix} is that
\[
\delta_{X_t,P} = 1-\gamma_{min}
\]
In particular, we may explicitly compute all of the terms of the matrix $\mathbf{\Sigma}^{(n)}$.

\begin{theorem}\label{completeCov}
Let $\mathbf{\Sigma}^{(n)}$ denote the covariance matrix of Theorem \ref{ahclt} for the simple random walk on $K_n$. Then, for any choice of spanning tree of $K_n$ paired with any choice of orientation of the extra edges $\{e_i = (x_i,y_i)\}$,
\begin{align}
  &\mathbf{\Sigma}^{(n)}(i,i) = \frac{n-2}{n^2\binom{n}{2}};\\
  \intertext{if the head of $e_i$ agrees with the tail of $e_j$, or
  vice versa, then}
  &\mathbf{\Sigma}^{(n)}(i,j) = \left(\frac{1}{2n\binom{n}{2}}\right)^2;\\
     \intertext{If the head (resp. tail) of $e_i$ agrees with the head
  (resp. tail) of $e_j$, then}\
  &\mathbf{\Sigma}^{(n)}(i,j) =
    -\left(\frac{1}{2n\binom{n}{2}}\right)^2;
     \intertext{If $e_i$ and $e_j$ do not share an endpoint, then}
&\mathbf{\Sigma}^{(n)}(i,j) = 0 \,.
\end{align}
\end{theorem}

\begin{proof}
Using Lemma \ref{stationaryFormula}, we see that $\pi_E(e_i) = \frac{1}{2n\binom{n}{2}}$. The theorem then follows from Lemma \ref{twoconstantmatrix} and Theorems \ref{SpectralBound} and \ref{offdiagonalbounds}.
\end{proof}

\begin{remark}
It is notable that the expressions of Theorem \ref{completeCov} are all given by rational functions in the parameter $n$. This is predicted by the theory of stochastic virtual relations on $\FI$-graphs (see \cite{RW,RSW,RW2}). This theory will return later when we compute the winding of two particles on a star graph in Section \ref{windEx}.
\end{remark}

\section{Winding of random walks in $\F_n(G)$}
\label{sec:Wind}

\subsection{Winding: the general case}

In this section, we apply the work of the previous sections to study random winding in tree configuration spaces. This work is heavily inspired by a large variety of classical studies of random winding of points in the plane \cite{Bel,BF,Sp}, as well as newer studies of related phenomena \cite{WT}.\\

We begin this section by establishing the primary random process of study on $\F_n(G)$. We then turn our attention to defining how we will encode winding as a (multivariate) statistic of this process.

\begin{definition}
Fix $n \geq 0$, as well as a leaf-rooted planar tree $G$, which is not a path. Further assume that $G$ satisfies the subdivision condition of Theorem \ref{cellularModel}. We define a random process on the state space
\[
\X_G := \{ \{x_1,\ldots,x_n\} \mid x_i \in V(G)\},
\]
in the following way. Given a configuration $\{x_1,\ldots,x_n\} \in \X_G$, one first flips a coin to determine whether anything will move. Assuming this first test passes, one then chooses an edge of $G$ uniformly. If this edge is not connected to any of the $x_i$, or if its two endpoints are both in the set $\{x_1,\ldots,x_n\}$, then the process holds in place. Otherwise, if exactly one of the end points of the edge is of the form $x_i$, then one replaces $x_i$ in $\{x_1,\ldots,x_n\}$ with the other endpoint of this edge.
\end{definition}

\begin{remark}
Note that we will not at any point be varying the number of points being configured, and therefore the lack of the parameter $n$ in $\X_G$ and $(X_t^G,P_G)$ should not cause confusion.\\

We also note that the above process in no way uses the fact that we have chosen an embedding of $G$ into the plane, nor does it use the fact that $G$ has been rooted at one of its leaves. These assumptions on $G$ are important in our ultimate definition of the winding process.
\end{remark}
 
One should observe that the Markov chain $(X_t^G,P_G)$ is both
connected and aperiodic. In the literature, this random process is
usually called the \emph{discrete exclusion process}.
Exclusion processes like the above have been extensively studied in
the literature in a variety of different forms from a variety of
different perspectives. See \ocite{Lig1} for an overview.
In this work we will implicitly make use of the following result,
conjectured by D.~Aldous and proven by \ocite{CLR}.

\begin{theorem}[\cite{CLR}]
The spectral gap of the exclusion process $(X_t^G,P_G)$ is equal to the spectral gap of the (lazy) simple random walk on the tree $G$.
\end{theorem}

For any $t \geq 0$, the sequence of configurations
\[
(X_0^G,X_1^G,\ldots,X_t^G)
\]
induces a path
\[
\rho_t(s):[0,1] \rightarrow \F_n(G).
\]
In the classical setting of two particles in the plane, one converts a random path into a winding statistic by charting the difference between the two points, and recording the total accumulated angle around the origin. In our setting we do not have access to these types of tools. Therefore, given such a path in $\F_n(G)$, our next goal will be to define a closure $\overline{\rho_t(s)}:S^1 \rightarrow \F_n(G)$, whose class in $H_1(\F_n(G))$ is completely determined by $\rho$.

Recall the discretized configuration space $D\F_n(G)$. The state space $\X_G$ agrees with the zero-skeleton of $D\F_n(G)$, and the Markov process $(X_t^G,P_G)$ is a model of a random walk on 	the one-skeleton of $D\F_n(G)$. Recall that we have assumed that $G$ is both planar and leaf-rooted. These two choices induce a well ordering on the vertices of $G$, as well as allow us to define the critical and collapsible $1$-cells of $D\F_n(G)$ (see Definition \ref{criticalCells}). To conclude, we fix the spanning tree of the one-skeleton of $D\F_n(G)$ induced by the collapsible $1$-cells, which notably does not contain the critical one-cells. We also fix an ordering of these critical cells, and orient them as in Definition \ref{criticalCells}. We will write these critical cells as $e_i = (x_i,y_i)$. Generally, we will use $(\widetilde{x_i},\widetilde{y_i})$ to denote a generic edge which lies on top of $e_i$, whenever the exact cell is unimportant.\\

For this section we will write $g$ for the number of critical one-cells of $D\F_n(G)$. Note that Theorem \ref{homologyBasis} tells us that the critical one-cells of $D\F_n(G)$ can be identified with a basis of $H_1(\F_n(G)) \cong \Z^g$.

\begin{definition}
Let $t \geq 0$, and write $\rho_t:[0,1] \rightarrow \F_n(G)$ for the path induced by the Markov chain $(X_0^G,\ldots,X_t^G)$. Then we define the \emph{closure} of $\rho_t$ to be the loop $\overline{\rho_t}:S^1 \rightarrow \F_n(G)$ defined by first performing $\rho_t$, and then performing the path through the aforementioned spanning tree of $D\F_n(G)$ from $X_t^G$ to $X_0^G$. We also write $[\overline{\rho_t}]$ to denote the element of $H_1(\F_n(G))$ induced by $\overline{\rho_t}$.\\

The \emph{winding of our Markov chain after $t$-steps} is then defined to be the (random) $g$-tuple
\[
W(t) := [\overline{\rho_t}] \in H_1(\F_n(G)) \cong \Z^g
\]
In particular, the winding of our chain can be equivalently thought of as a kind of accumulated homology statistic on $D\F_n(G)$, where all non-critical 1-cells that are excluded from our spanning tree are either ignored, or counted along side a unique critical edge in accordance with Theorem \ref{howCellsWork}.
\end{definition}

As with our section on accumulated homology, one of our interests in this section will be in understanding the limiting distribution of
$
\lim_{t \to \infty} W(t) \stackrel{\mathcal{D}}{\rightarrow} \larger{?}
$
Before we do this, however, we take a moment to note that the setup of our problem lends itself to another natural question.\\

The chosen embedding of $G$ into the plane induces a map of topological spaces
\[
\F_n(G) \hookrightarrow \F_n(\R^2),
\]
which in turn induces a map
\[
H_1(\F_n(G)) \rightarrow H_1(\F_n(\R^2)).
\]
It is well known that $H_1(\F_n(\R^2)) \cong \Z$ (see \cite{Arn}, for instance), and it is interesting to ask what the images are of our chosen basis vectors of $H_1(\F_n(G))$ under this map. In fact, it can be shown that if $b \in H_1(\F_n(G))$ is a basis vector corresponding to some critical $1$-cell, then
\[
b \mapsto \pm 1
\]
Indeed, this follows from work of Farley \ycite{Fa}, as well as An,
Drummond-Cole, and Knudsen \ycite{ADK} which show that the vector $b$
can be expressed topologically by what is known as a \emph{star
  move}.  See also \ocite{CL}. This inspires the following definition.

\begin{definition}\label{planarDef}
The \emph{planar winding our Markov chain after $t$-steps} is defined to be the (random) integer
\[
PW(t) := [\overline{\rho_t}] \in H_1(\R^2) \cong \Z.
\]
Equivalently, $PW(t)$ can be written as
\[
PW(t) = \sum_{i = 1}^g \epsilon_i W(t)_i
\]
where $\epsilon_i \in \{\pm 1\}$ is determined by the embedding of $G$ into the plane, as well as the choice of orientation of the critical one-cells.
\end{definition}

Having established the various definitions, we are now ready to state our main results.

\begin{theorem}\label{windingbound}
There exists a $g \times g$ matrix $\mathbf{\Sigma}$ such that,
\[
\lim_{t \to \infty} \frac{W(t)}{\sqrt{t}} \stackrel{\mathcal{D}}{\rightarrow} N(0,\mathbf{\Sigma})
\]
where $N(0,\mathbf{\Sigma})$ is the multivariate normal distribution with covariance matrix $\mathbf{\Sigma}$. Moreover, one has
\begin{align}
\frac{(\# e_i)}{\binom{E+1}{n}E^2}\Biggl(E - \frac{1}{\delta_{X_t,P}}\Biggr) &
                                                                   \leq  \mathbf{\Sigma}(i,i) \leq \frac{(\# e_i)}{\binom{E+1}{n}E^2}\Biggl(E - \frac{1}{1-\gamma_{min}}\Biggr)
\end{align}

\begin{multline}
\frac{(\# e_i) + (\# e_j)}{2(1-\gamma_{\min})\binom{A}{n}(A-1)^2} -
  \frac{\|\widetilde{f_i} + \widetilde{f_j}\|_\pi^2}{\delta_{X^G_t,P_G}}
                                                                  \\
                                                                 \leq
                                                                  \mathbf{\Sigma}(i,j)
                                                                  \leq \frac{\# e_i + \# e_j}{2\delta_{X^G_t,P_G}\binom{A}{n}(A-1)^2} - \frac{\|\widetilde{f_i} + \widetilde{f_j}\|_\pi^2}{1-\gamma_{\min}}
\end{multline}
where $E$ is the number of edges of $G$, $(\# e_i)$ is the number of edges lying above the critical edge $e_i$, $\gamma_{min}$ is the smallest eigenvalue of $P_G$, and $\widetilde{f_i}:\X_G \rightarrow \R$ is the statistic which assumes the value $\frac{1}{2(E-1)}$ at vertices of the form $\widetilde{x_i}$, $-\frac{1}{2(E-1)}$ at vertices of the form $\widetilde{y_i}$, and 0 elsewhere.
\end{theorem}

\begin{proof}

It is not the case in general that the winding statistic is the accumulated homology on $D\F_n(G)^{(1)}$, because of the edges which lie over the critical edges. That being said, Theorem \ref{howCellsWork} implies that winding is a projection of this accumulated homology. Moreover, because the cells lying over a given critical edge are disjoint, the sum of signed indicator functions do not create covariance terms when applying the computations of Theorem \ref{SpectralBound}. Our theorem is then just a simple consequence of Theorem \ref{SpectralBound}, where one replaces the indicator functions by sums of indicator functions of edges lying over critical edges.

\end{proof}

Using the relationship between our winding statistic, and the planar winding statistic, we also can conclude the following.

\begin{theorem}
If $G$ has maximum vertex degree 3 then
\[
\lim_{t \to \infty} \frac{PW(t)}{\sqrt{t}} \stackrel{\mathcal{D}}{\rightarrow} N(0,\sigma_G^2),
\]
where $N(0,\sigma_G^2)$ is the standard centered normal distribution with variance $\sigma_G^2$. Moreover, one has,
\[
\frac{N}{E\binom{E+1}{n}} - \frac{2\|\sum_i \epsilon_i\widetilde{f_i}\|_{\pi}^2}{\delta_{X^G_t,P_G}} \leq \sigma^2_G \leq \frac{N}{E\binom{E+1}{n}} - \frac{2\|\sum_i \epsilon_i\widetilde{f_i}\|_{\pi}^2}{1-\gamma_{min}}
\]
where $N$ is the number of non-collapsible 1-cells of $D\F_n(G)$.
\end{theorem}

We note that this kind of bounded winding in the plane was studied in for two Brownian particles constrained within an annulus in \cite{WT}. Their results are very similar to ours.\\

\subsection{Winding: the two particle case}

In this section, we considering the winding process of two non-overlapping particles in a tree. Limiting the number of particles in this way simplifies exposition considerably, as we never have to worry about the condition of Theorem \ref{cellularModel}. This case is also significant due to its parallels with the more classical setting of two particles in the plane \cite{Sp,WT}.\\

\emph{As in the previous section, we will now fix for all time a tree $G$, an embedding of $G$ in the plane, as well as a leaf to be the root of $G$.}

\begin{definition}
For a graph $\Gamma$, we write $\binom{\Gamma}{2}$ to denote the graph whose vertices are indexed by unordered pairs of vertices of $\Gamma$, and whose edge relation is given by
\begin{align*}
\{v,w\} &\sim \{u,w\}, \text{ where $\{v,u\} \in E(\Gamma)$}\\
\{v,w\} &\sim  \{v,u\}, \text{ where $\{w,u\} \in E(\Gamma)$}
\end{align*}
\end{definition}

If $G$ is a planar leaf-rooted tree, then we see that $\binom{G}{2}$ is precisely the one-skeleton of $D\F_2(G)$. In particular, our random process of two particles moving in $\F_2(G)$ may be equivalently thought of as a model of a random walk on $\binom{G}{2}$. As we saw in the previous section, Theorems \ref{cellularModel} and \ref{homologyBasis} then imply that the accumulated homology of this walk on $\binom{G}{2}$, with respect to the spanning tree of collapsible $1$-cells, precisely encodes the winding of our original process in $\F_2(G)$.

Recall that we write $(X^G_t,P_G)$ for the exclusion process on $\binom{G}{2}$.\\

The major conjecture of this section will suggest that the winding of our random process is robust enough to recover the tree on which the particles are moving. 

\begin{conjecture}[The Strong Conjecture]\label{mainconj}
Assume that $G$ has no vertices of degree 2, and write $\mathbf{\Sigma}_G$ for the covariance matrix determined by the large $t$ behavior of the winding statistic of the exclusion process on $\binom{G}{2}$, as in Theorem \ref{windingbound}. If $G'$ is another planar leaf-rooted tree with no vertices of degree $2$ for which $\Sigma_G = \Sigma_{G'}$, then $G$ and $G'$ are isomorphic as planar leaf-rooted trees.
\end{conjecture}

In the next section, we provide evidence for this conjecture by showing that the covariance matrices associated to the two trees of Figure \ref{samedegseq} are distinct. We also note that there is a weaker version of Conjecture \ref{mainconj}, whose affirmation would still be significant from the perspective of graph configuration spaces.

\begin{conjecture}[The Weak Conjecture]
Assume that $G$ has no vertices of degree 2, and write $\mathbf{\Sigma}_G$ for the covariance matrix determined by the large $t$ behavior of the winding statistic of the exclusion process on $\binom{G}{2}$, as in Theorem \ref{windingbound}. If $G'$ is another planar leaf-rooted tree with no vertices of degree $2$ for which $\Sigma_G = \Sigma_{G'}$, then $G$ and $G'$ have the same degree sequence.
\end{conjecture}

 To finish this section, we give a heuristic justification for why Conjecture \ref{mainconj} is natural, as well as why it is significant from the perspective of graph configuration spaces.

Theorem \ref{homologyBasis} can be used (see \cite{R}, for instance) to show that the homotopy type of $\F_2(G)$ is determined entirely by the degree sequence of $G$. In fact, it is determined by certain combinatorial data, which is a considerably weaker numerical invariant of $G$ than its degree sequence.  In particular, homotopy theoretic invariants of $\F_2(G)$ cannot recover $G$ from $\F_2(G)$. To see why this is the case, recall that the discrete Morse theoretic approach of \cite{FS,Fa} implies that $H_1(\F_2(G))$ has a basis in bijection with certain star moves on the tree. In particular, the homology of $\F_2(G)$ can, at most, see the number of vertices of degree $\geq 3$, as well as the degrees of these vertices. On the other hand, the exclusion process being performed on $G$ imposes a kind of motion to these homology classes: if two adjacent branching vertices contain an imbalanced number of vertices on either side of them , then you expect the star moves being performed at either to have some correlation. In particular, the matrix $\Sigma_G$ should, in principal, contain the information of both the degree sequence of $G$, as well as global information about how vertices of degree at least 3 are distributed in the tree. From these two pieces of information, one should be able to recover the tree itself.

\begin{figure}
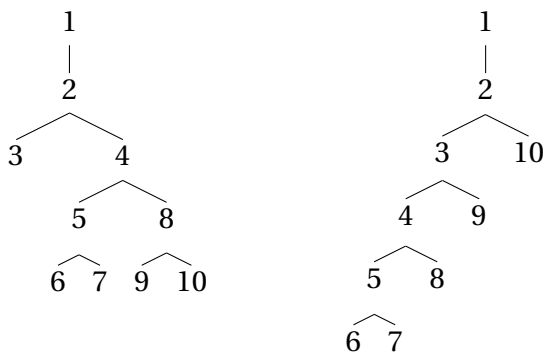

\Tree[.1 [.2 [.3 ] [.4 [.5 [.6 ] [.7 ]] [.8 [ .9 ] [.10 ] ] ]]] \Tree[.1 [.2 [.3 [.4 [.5 [.6 ] [.7 ]] [.8 ]] [.9 ]] [.10 ]]]

\caption{The smallest example of two non-isomorphic trees with the same degree sequence and no vertices of degree 2.}\label{samedegseq}
\end{figure}

\subsection{Winding: examples} \label{windEx}

In this section, we take the time to display some worked examples of winding. To begin, we compute the covariance matrix of the exclusion process winding of two particles traversing the \emph{star graph with $l$-leaves}. This is the tree with $l$ vertices of degree 1, called the \emph{leaves}, and one vertex of degree $l$, called the \emph{center}, and we denote it by $G_l$. Following this abstract computation, we complete the argument that the winding covariance matrices are distinct for the two particle walks on the planar leaf-rooted trees of Figure \ref{samedegseq}, up to reliance on numerical data.

While our ultimate goal is to compute the winding covariance matrix of the two particle exclusion process on the star graph, we will begin by bounding the entries of this matrix via the spectrum of $P_{G_l}$. Note that, in so far as winding is concerned, the compliment of the collapsible 1-cells of $\binom{G_l}{2}$ are precisely the critical 1-cells. In particular, winding in this case is literally the accumulated homology of the simple random walk on $\binom{G_l}{2}$.

To begin our computation, note that the vertices of the associated graph $\binom{G_l}{2}$ can be partitioned into two types: Those corresponding to configurations where one particle is in the center, and those corresponding to configurations where both particles are on leaves. In particular, it will be useful going forward to identify the vertices of $\binom{G_l}{2}$ with the leaves of $G_l$ that are being occupied in the associated configuration. If we organize our basis of $\R V(\binom{G_l}{2})$ by listing the configurations with a single occupied leaf first, followed by the those with two leaves, the matrix of $P_{G_l}$ assumes the block form,
\[
 P_{G_l} = \begin{pmatrix}
\frac{l+1}{2l} I_{l} & A\\
A^\dagger & \frac{l-1}{l} I_{\binom{l}{2}}
\end{pmatrix},
\]
where $A$ is the $(l \times \binom{l}{2})$-matrix given by
\[
A(i, \{j,k\}) = \begin{cases} 0 &\text{ if $i \notin \{j,k\}$}\\ \frac{1}{2l} &\text{ otherwise.}\end{cases}
\]
Our goal is to therefore compute the determinant of the matrix
\begin{align}
P_{G_l} - \lambda I_{l + \binom{l}{2}} = \begin{pmatrix}
(\frac{l+1}{2l} - \lambda) I_{l} & A\\
A^\dagger & (\frac{l-1}{l} - \lambda) I_{\binom{l}{2}}\end{pmatrix}\label{blockMatrix}
\end{align}
We begin with the following standard facts from linear algebra.

\begin{lemma}\label{blockdet}
Let $M$ be an $(n+m) \times (n+m)$ matrix, written in block form as
\[
M = \begin{pmatrix} A & B\\ C & D \end{pmatrix},
\]
where $A$ is $n \times n$, $B$ is $n \times m$, $C$ is $m \times n$, and $D$ is $m \times m$. Then, if $D$ is invertible, one has
\[
\det(M) = \det(A - BD^{-1}C)\det(D)
\]
\end{lemma}

These are the two technical tools we need to complete our desired computation.

\begin{theorem}\label{starspec}
With notation as above, one has for all $l \geq 3$, the spectrum of $P_{G_l}$ is given by,
\begin{align}
1 &\text{ with multiplicity 1;}\\
1 - \frac{1}{2l} & \text{ with multiplicity $l-1$;}\\
1 - \frac{1}{l} & \text{ with multiplicity $\binom{l}{2} - l$;}\\
\frac{1}{2} &\text{ with multiplicity $l-1$;}\\
\frac{1}{2} - \frac{1}{2l} &\text{ with multiplicity 1.}
\end{align}
\end{theorem}

\begin{proof}
Using (\ref{blockMatrix}) and Lemma \ref{blockdet}, we have
\[
\det(P_{G_l} - \lambda I_{l + \binom{l}{2}}) = \det((\frac{l+1}{2l} - \lambda) I_{l} - \frac{1}{\frac{l-1}{l} - \lambda}AA^{\dagger})(\frac{l-1}{l} - \lambda)^{\binom{l}{2}}
\]

One easily computes that $A^\dagger A$ is the $l \times l$ matrix with entries given by
\[
(A^{\dagger}A)(i,j) = \begin{cases} \frac{1}{4l^2} & \text{ if $i \neq j$}\\ \frac{l-1}{4l^2} &\text{ otherwise.}\end{cases}
\]
In particular, $(\frac{l+1}{2l} - \lambda) I_{l} - \frac{1}{\frac{l-1}{l} - \lambda}AA^{\dagger}$ is the matrix whose diagonal terms are
\[
\frac{1}{4l^2(\frac{l-1}{l} - \lambda)}((4l^2(\frac{l-1}{l} - \lambda)(\frac{l+1}{2l} - \lambda) - (l-1)),
\]
and whose off diagonal terms are
\[
\frac{-1}{4l^2(\frac{l-1}{l} - \lambda)}.
\]
Applying Lemma \ref{twoconstantmatrix}, we conclude
\begin{align*}
\det(P_{G_l} - \lambda I_{l + \binom{l}{2}}) = \frac{((4l^2(\frac{l-1}{l} - \lambda)(\frac{l+1}{2l} - \lambda) + 2-l)^{l-1}(2-2l + (4l^2(\frac{l-1}{l} - \lambda)(\frac{l+1}{2l} - \lambda))(\frac{l-1}{l} - \lambda)^{\binom{l}{2} - l}}{(4l^2)^l}
\end{align*}

Computer algebra may then be used to find all roots of this polynomial, as desired.
\end{proof}

Combining Theorem \ref{starspec}, as well as Theorem \ref{windingbound}, we obtain the following.

\begin{theorem}
There exists a $\binom{l-1}{2} \times \binom{l-1}{2}$ matrix $\mathbf{\Sigma}$ such that,
\[
\lim_{t \to \infty} \frac{W(t)}{\sqrt{t}} \stackrel{\mathcal{D}}{\rightarrow} \Norm(0,\mathbf{\Sigma})
\]
where $\Norm(0,\mathbf{\Sigma})$ is the multivariate normal distribution with covariance matrix $\mathbf{\Sigma}$. The entries of $\mathbf{\Sigma}$ satisfy the following inequalities
\begin{align}
 \mathbf{\Sigma}(i,i) & \leq  \frac{l-1}{l(l+1)\binom{l+1}{2}}\\
-\frac{3l+1}{l(l+1)\binom{l+1}{2}}  \leq \mathbf{\Sigma}(i,j) & \leq \frac{2l-1}{l(l+1)\binom{l+1}{2}} 
\end{align}
\end{theorem}

Having completed our bounds, we next turn our attention to more explicitly computing the covariance via computing the discrete Green's function $(I-P_{G_l})^{-1}$. Our approach here follows the general computational approach of Chung and Yau \ycite{CY}, which expresses the entries of the Green's function as sums of  hitting times. We also make use of the algebraic tools of~\ocite{RW2}. These tools essentially derive from the fact that there is an action of the symmetric group $S_l$ on the leaves of $G_l$.

\begin{theorem}\label{hittingtimes}
For each pair $x,y \in V(\binom{G_l}{2})$, write $Q(x,y)$ for the expected hitting time of the exclusion process between $x$ and $y$. Then,
\begin{align}
&Q(\{1,2\},1) =l^2+2l,\hspace{.25cm}  Q(\{2,3\},1) = 2l^2 +3l,\hspace{.25cm}  Q(2,1) = 2l^2+2l,\\ \label{roof1}
&Q(1,\{1,2\}) = \frac{l^3+l^2-2l}{2},\hspace{.25cm} Q(3,\{1,2\}) = \frac{l^3+3l^2}{2},\\
&Q(\{1,3\},\{1,2\}) = \frac{l^3+2l^2+l}{2},\hspace{.25cm} Q(\{3,4\},\{1,2\}) = \frac{l^3+3l^2+2l}{2}
\end{align}
\end{theorem}

\begin{proof}
We begin by noting that the aforementioned seven computations determine all hitting times on $\binom{G_l}{2}$ by application of the symmetric group action. 

We will illustrate the method to compute the hitting times (\ref{roof1}). Let $\widehat{P}$ denote the matrix
\[
\widehat{P} =
\begin{pmatrix}
 \frac{l-1}{l}& 0 & \frac{1}{2l}\\
 0 & \frac{l-1}{l} & \frac{1}{l}\\
 \frac{1}{2l}& \frac{l-2}{2l} & \frac{l+1}{2l}
 \end{pmatrix}
 \]
 In the language of \cite{RW2}, this is a principal minor of the transition matrix of the walk associated to the 1-roofed orbits of $G_l$. In particular, the vector $Q = \begin{pmatrix} Q(\{1,2\},1)\\ Q(\{2,3\},1)\\ Q(2,1)\end{pmatrix}$ is the unique solution to the matrix equation
 \[
 (I-\widehat{P})Q = \begin{pmatrix} 1\\1\\1 \end{pmatrix}
 \]
 This can be solved in any computer algebra system to obtain the desired results.
 \end{proof}
 
 Having computed these hitting times, we may now complete our description of the discrete Green's function associated to the two particle exclusion process on $G_l$.
 
 \begin{theorem}
Write for any pair of vertices $x,y \in V(\binom{G_l}{2})$, write
$\mathcal{G}(x,y) := (I-P_{G_l})^{-1}(x,y)$.
Let
\begin{align*}
  T_1(\ell) & = \frac{2 (\ell - 1) (2 \ell + 1) (\ell + 2)}{\ell (\ell
              + 1)^2} \\
  T_2(\ell) & = \frac{\ell^4 + 4\ell^3 -  \ell^2 - 12\ell    + 12}{\ell
              (\ell +  1)^2}
\end{align*}
Then,
\begin{align*}
  \G(1,1) & = T_1(\ell)
  &
   \G(\{1,2\},\{1,2\}) & = T_2(\ell) \\
  \G(\{1,2\},1) & =T_1(\ell)   - \frac{\ell^2+2\ell}{\binom{\ell+1}{2}}
  &
  \G(\{2,3\},1) & = T_1(\ell) - \frac{2\ell^2
                 +3\ell}{\binom{\ell+1}{2}}\\
  \G(2,1) & = T_1(\ell) -\frac{2\ell^2+2\ell}{\binom{\ell+1}{2}}
  &
 \G(1,\{1,2\}) & = T_2(\ell) -
                 \frac{\ell^3+\ell^2-2\ell}{2\binom{\ell+1}{2}} \\
  \G(3,\{1,2\}) & =T_2(\ell) - \frac{\ell^3+3\ell^2}{2\binom{\ell+1}{2}}
          &
   \G(\{1,3\},\{1,2\}) & = T_2(\ell) - \frac{\ell^3+2\ell^2+l}{2\binom{\ell+1}{2}}\\
\end{align*}

\end{theorem}
\begin{proof}
Once again we note that these nine computations determine the discrete Green's function at every pair of vertices because of the symmetric group action.\\

Finally, the above computations follow from Theorem
\ref{hittingtimes}, as well as the formulas of \ocite{CY}
\begin{align*}
\mathcal{G}(x,x) = \pi(x)^2\sum_z Q(z,x)\\
\mathcal{G}(x,y) = \mathcal{G}(y,y)-\pi(x)Q(x,y)
\end{align*}
\end{proof}

To conclude our computations, we prove that the covariance matrices associated to the planar leaf-rooted trees of Figure \ref{samedegseq} can distinguish these trees. For what follows, we write $G_1$ to denote the left-most planar leaf-rooted tree in Figure \ref{samedegseq}, and $G_2$ to denote the right-most tree. Our first result essentially says that the tree $G_1$ is too symmetric for there to be non-trivial correlation between distinct star-moves.

\begin{proposition}
The Covariance matrix $\mathbf{\Sigma}_{G_1}$ is a diagonal matrix.
\end{proposition}

\begin{proof}
The matrix $\mathbf{\Sigma}_{G_1}$ is $4 \times 4$, with rows in columns indexed by the vertices labeled 2,4,5, and 8. If $i$ is any-such index, and $j \neq 4$ is another, we claim that $\mathbf{\Sigma}_{G_1}(i,j) = 0$. Indeed, this follows from the fact that there is an automorphism of $G_1$, e.g. the one which switches the two leaves connected to $j$, which negates the basis vector associated to $j$, while leaving the basis vector associated to $i$ fixed.
\end{proof}

On the other hand, this symmetry is non-existent in the tree $G_2$. Labeling the rows and columns of $\mathbf{\Sigma}_{G_2}$ by the vertices $2,3,4,$ and $5$, numerical simulation has shown that, with high probability, one has $\mathbf{\Sigma}_{G_2}(3,4) \neq 0$.

\section*{Acknowledgements}
The authors would like to send thanks to Aaron Abrams, Gabor Lippner, 
Nick Proudfoot, and Graham White for useful conversations. Some of the
ideas in this paper, especially the main conjecture \ref{mainconj},
were discussed at the American Institute of Mathematics' workshop on
graph configuration spaces. We thank AIM for providing these
accommodations.



\end{document}